\documentclass[12pt]{amsart}
\usepackage{amsthm}
\usepackage{amssymb}
\usepackage{longtable}
\usepackage{geometry}
\usepackage{enumerate}
\usepackage{listings}
\usepackage{setspace}
\usepackage{upgreek}
\usepackage{tikz}
\usepackage{tikz-cd}
\usetikzlibrary{matrix, arrows}
\usepackage[unicode=true]
 {hyperref}
\hypersetup{
colorlinks=true,
urlcolor=black,
citecolor=blue,
linkcolor=blue,
}

\allowdisplaybreaks

\newtheorem{thm}{Theorem}[section]
\newtheorem{prop}[thm]{Proposition}
\newtheorem{lem}[thm]{Lemma}

\theoremstyle{definition}
\newtheorem{definition}[thm]{Definition}

\newtheorem{example}[thm]{Example}
\newtheorem{question}[thm]{Question}
\theoremstyle{remark}
\newtheorem{remark}[thm]{Remark}

\numberwithin{equation}{section}

\DeclareMathOperator*{\lcm}{lcm} 
\newcommand{\fa}{\mathfrak{a}}
\newcommand{\fv}{\mathfrak{v}}
\newcommand{\fm}{\mathfrak{m}}
\newcommand{\ep}{\epsilon}
\newcommand{\bZ}{\mathbb{Z}}
\newcommand{\bN}{\mathbb{N}}
\newcommand{\Perm}{\mathrm{Perm}}

\newcommand{\Hol}{\mathrm{Hol}}
\newcommand{\NHol}{\mathrm{NHol}}
\newcommand{\Aut}{\mathrm{Aut}}
\newcommand{\Norm}{\mathrm{Norm}}
\newcommand{\HH}{\mathcal{H}}
\newcommand{\Id}{\mathrm{Id}}
\newcommand{\mmod}{\hspace{-2.5mm}\pmod}

\begin{document}

\large 

\title[Multiple holomorph of groups of squarefree or odd prime power order]{On the multiple holomorph of groups \\of squarefree or odd prime power order}

\author{Cindy (Sin Yi) Tsang}
\address{School of Mathematics (Zhuhai), Sun Yat-Sen University}
\email{zengshy26@mail.sysu.edu.cn}\urladdr{http://sites.google.com/site/cindysinyitsang/} 

\date{\today}

\maketitle

\begin{abstract}Let $G$ be a group and write $\Perm(G)$ for its symmetric group. Define $\Hol(G)$ to be the holomorph of $G$, regarded as a subgroup of $\Perm(G)$, and let $\NHol(G)$ denote its normalizer. The quotient $T(G) = \NHol(G)/\Hol(G)$ has been computed for various families of groups $G$, and in most of the known cases, it turns out to be elementary $2$-abelian, except for two groups of order $16$ and some groups of odd prime power order and nilpotency class two. In this paper, we shall show that $T(G)$ is elementary $2$-abelian for all finite groups $G$ of squarefree order, and that $T(G)$ is not a $2$-group for certain finite $p$-groups $G$ of nilpotency class at most $p-1$. %For us, the trivial group is regarded as elementary abelian. 
%a finite non-nilpotent group $G$ for which the quotient $T(G)$ is not elementary $2$-abelian or not even a $2$-group?
\end{abstract}

\tableofcontents

\section{Introduction}

Let $G$ be a group and write $\Perm(G)$ for its symmetric group. Recall that a subgroup $N$ of $\Perm(G)$ is said to be \emph{regular} if the map
\begin{equation}\label{xi} \xi_N : N\longrightarrow G;\hspace{1em} \xi_N(\eta) = \eta(1_G)\end{equation}
is bijective, or equivalently, if the $N$-action on $G$ is both transitive and free. For example, both $\lambda(G)$ and $\rho(G)$ are regular subgroups of $\Perm(G)$, where
\[\begin{cases}
\lambda:G\longrightarrow\Perm(G);\hspace{1em}\lambda(\sigma) = (x\mapsto \sigma x)\\
\rho:G\longrightarrow\Perm(G);\hspace{1em}\rho(\sigma) = (x\mapsto x\sigma^{-1})
\end{cases}\]
denote the left and right regular representations of $G$, respectively. Further, recall that the \emph{holomorph of $G$} is defined to be
\[\label{Hol(G)} \Hol(G) = \rho(G) \rtimes \Aut(G).\]
Alternatively, it is not hard to verify that
\[\Norm_{\Perm(G)}(\lambda(G)) = \Hol(G) = \Norm_{\Perm(G)}(\rho(G)).\]
Then, it seems natural to ask whether $\Perm(G)$ has other regular subgroups whose normalizer is also equal to $\Hol(G)$. This problem was first considered by G. A. Miller \cite{Miller}. More specifically, put
\[\HH_0(G) = \left\lbrace\begin{array}{c}\mbox{regular subgroups $N$ of $\Perm(G)$ isomorphic to $G$}\\\mbox{and such that $\Norm_{\Perm(G)}(N) = \Hol(G)$}\end{array}\right\rbrace.\]
In \cite{Miller}, he defined the \emph{multiple holomorph of $G$} to be
\begin{equation}\label{NHol def} \NHol(G) = \Norm_{\Perm(G)}(\Hol(G)),\end{equation}
which clearly acts on $\HH_0(G)$ via conjugation, and he showed that this action is transitive. Hence, the quotient group
\[ T(G) = \frac{\NHol(G)}{\Hol(G)}\]
%\[ T(G) = \NHol(G)/\Hol(G)\]
acts regularly on and so has the same cardinality as $\HH_0(G)$. In fact, we have 
\begin{equation}\label{HH0'}\HH_0(G) = \{\pi\lambda(G)\pi^{-1} : \pi\in\NHol(G)\},\end{equation}
and in the case that $G$ is finite, we also have
\begin{equation}\label{HH0}\HH_0(G) = \{N\lhd\Hol(G): N\simeq G\mbox{ and $N$ is regular}\}.\end{equation}
These facts are easy to show, and a proof may be found in \cite[Section 2]{Tsang NHol} or \cite[Section 1]{Kohl NHol}, for example. Let us mention in passing that regular subgroups of the holomorph, not just the ones which are normal, have close connections with Hopf-Galois structures on field extensions and set-theoretic solutions to the Yang-Baxter equation; see \cite[Chapter 2]{Childs book} and \cite{Skew braces}.

\vspace{1.5mm}

The study of $T(G)$ did not attract much attention initially other than \cite{Miller} and \cite{Mills}. But recently, T. Kohl, who was originally interested in Hopf-Galois structures, revitalized this topic of research in \cite{Kohl NHol}. His work then motivated the calculation of $T(G)$ for other groups $G$. Note that for non-abelian $G$, we have $T(G)\neq1$ because the element $\pi_{-1}\Hol(G)$ is non-trivial, where $\pi_{-1}$ is the involution permutation $x\mapsto x^{-1}$ on $G$. It turns out that $T(G)$ is elementary $2$-abelian (including the trivial group) in all of the following five cases.
\begin{itemize}
\item finitely generated abelian groups $G$; see \cite{Miller,Mills,Caranti1}.
\item finite dihedral groups $G$ of order at least $6$; see \cite{Kohl NHol}.
\item finite dicyclic groups $G$ of order at least $12$; see \cite{Kohl NHol}.
\item finite perfect groups $G$ with trivial center; see \cite{Caranti2}.
\item finite quasisimple or almost simple groups $G$; see \cite{Tsang NHol}.
\end{itemize}
Nevertheless, there are examples of groups $G$ for which $T(G)$ is not elementary $2$-abelian. In \cite[Section 3]{Kohl NHol}, it was noted that $T(G)$ is non-abelian (but still a $2$-group) for two groups $G$ of order $16$. Also, in \cite[Proposition 3.1]{Caranti3}, it was shown that $T(G)$ is not even a $2$-group for a large subfamily of
\begin{itemize}
\item finite groups of odd prime power order and nilpotency class two.
\end{itemize}
In this paper, we shall consider two new families of finite groups $G$. First, in Section~\ref{section1}, we shall show that $T(G)$ is elementary $2$-abelian for
\begin{itemize}
\item finite groups $G$ of squarefree order.
\end{itemize}
This was motivated by \cite{Byott squarefree}, in which all cyclic regular subgroups of $\Hol(G)$ for $G$ of squarefree order were enumerated, and these correspond to Hopf-Galois structures on cyclic field extensions of squarefree degree. Then, in Section~\ref{section2}, we shall show that $T(G)$ is not a $2$-group for a subfamily of
\begin{itemize}
\item finite $p$-groups $G$ of nilpotency class at most $p-1$.
\end{itemize}
We shall state the precise statements of our results in Theorems~\ref{main thm} and~\ref{main thm'} below. We chose not to include them here because more notation would have to be introduced. 

\vspace{1.5mm}

In view of the known results so far, it seems natural to ask whether there is a finite non-nilpotent group $G$ for which $T(G)$ is not elementary $2$-abelian or not even a $2$-group. By calculating the size of the set in (\ref{HH0}), which equals the order of $T(G)$, we have checked  in \textsc{Magma} \cite{magma} that such a group $G$ does exist. However, among the non-nilpotent groups $G$ of order at most $95$, the quotient $T(G)$ is not a $2$-group only for $G = \mbox{{\sc SmallGroup}}(n,i)$ \cite{SGlibrary}, where
\begin{equation}\label{(n,i)}(n,i)\in\{(48,12),(48,14),(63,1),(80,12),(80,14)\},\end{equation}
and they are all solvable. Let us end with the following question.

\begin{question}Is $T(G)$ a $2$-group for all finite insolvable groups $G$?
\end{question}

It would also be interesting to find out whether $T(G)$ is a $2$-group for most finite solvable (non-nilpotent) groups $G$, and if so, to give an explanation. %These are all solvable. Any insolvable group $G$ for which $T(G)$ is not a $2$-group?

%\begin{question}Does there exist a finite non-nilpotent group $G$ for which the quotient $T(G)$ is not elementary $2$-abelian or not even a $2$-group?
%\end{question}

\section{Groups of squarefree order}\label{section1}

Throughout this section, let $G$ denote a finite group of squarefree order. It is known, by \cite[Lemma 3.5]{squarefree}, that $G$ admits a presentation
\[ G = G(d,e,k) = \langle\sigma,\tau:\sigma^e=1_G,\,\tau^d = 1_G,\, \tau\sigma\tau^{-1}= \sigma^k\rangle,\]
where $d,e\in\bN$, and $k\in\bN$ is coprime to $e$ whose multiplicative order mod $e$ is equal to $d$. Note that $|G| = de$, which is squarefree by assumption, so 
\[ \gcd(d,e) = 1 \mbox{ and both }d,e\mbox{ are squarefree}.\]
Following \cite[Proposition 3.5]{Byott squarefree}, let us write
\[ |G| = de = d(gz),\mbox{ where }z = \gcd(e,k-1)\mbox{ and } g = e/z.\]
Observe that by definition and the fact that $\gcd(g,z)=1$, we have 
\[ k\equiv1\mmod{z},\mbox{ and }k\not\equiv1\mmod{q}\mbox{ for all primes $q\mid g$}.\]
For each divisor $f$ of $g$, define $d_f$ to be the multiplicative order of $k$ mod $f$. Since $\gcd(g,z)=1$ and $z$ divides $k-1$, we have $d = d_g$. Also, let us define
\[\mathfrak{o}(G) = \{(f_1,f_2)\in\bN^2 : g=f_1f_2\mbox{ and }\gcd(d_{f_1},d_{f_2}) \leq 2\}.\]
Note that this set depends only on the values of $g$ and $k$. We shall prove:

\begin{thm}\label{main thm}The group $T(G)$ is elementary $2$-abelian of order $\#\mathfrak{o}(G)$.
\end{thm}

\begin{remark}The number $\#\mathfrak{o}(G)$ is indeed a power of $2$. To see why, define a graph $\Gamma(G)$ as follows.
\begin{itemize}
\item The vertex set of $\Gamma(G)$ is the set of distinct prime factors of $g$.
\item Two vertices $q,q'$ are joined by an edge if and only if $\gcd(d_{q},d_{q'})\geq3$.
\end{itemize}
Since $g$ is squarefree, for any divisor $f$ of $g$, we have $d_{f}=\lcm_{q\mid f}d_q$. Hence, a pair $(f_1,f_2)\in\bN^2$ with $g=f_1f_2$  lies in $\mathfrak{o}(G)$ if and only if for each connected component $\Gamma_0$ of $\Gamma(G)$, the primes in $\Gamma_0$ either all divide $f_1$ or all divide $f_2$. We then deduce that $\#\mathfrak{o}(G) = 2^{\mathfrak{c}(G)}$, where $\mathfrak{c}(G)$ is the number of connected components in $\Gamma(G)$.%Note that $\kappa(G)$ is at most the number of distinct prime factors of $g$.
\end{remark}

To prove Theorem~\ref{main thm}, recall that $T(G)$ has the same size as $\HH_0(G)$. Below, we shall use (\ref{HH0}) to give a description of the elements of $\HH_0(G)$ in terms of certain congruence conditions. %They would allow us to compute the order of $T(G)$, and together with (\ref{HH0'}), to determine the group structure of $T(G)$.

\subsection{Preliminaries} Let $U(e)$ denote the set of integers $s\in\bZ$ coprime to $e$ and write $\varphi(\cdot)$ for Euler's totient function.

\begin{lem}\label{Aut lem}We have
\[ |\Aut(G)| = g\varphi(e)\mbox{ and }\Aut(G) = \langle\theta\rangle\rtimes\{\phi_s\}_{s\in U(e)},\]
where $\theta$ and $\phi_s$ are the automorphisms on $G$ determined by
\[ \theta(\sigma)=\sigma,\,\ \theta(\tau) = \sigma^z\tau,\,\ \phi_s(\sigma) =\sigma^s,\,\ \phi_s(\tau) = \tau.\]
Moreover, we have the relations
\[ \theta^g = \mathrm{Id}_G,\,\ \phi_s\phi_t = \phi_{st},\,\ \phi_s\theta\phi_s^{-1} = \theta^s.\]
\end{lem}
\begin{proof}See the proof of \cite[Lemma 4.1]{Byott squarefree}.
\end{proof}

%Observe that $G$ as well as $\Aut(G)$ and $\Hol(G)$ are all semidirect products. It is easy to see that we have
It is easy to see that we have the relations
\begin{align*}
\mbox{in $G$:}\hspace{1.05cm}\tau^b\sigma^a &= \sigma^{ak^b}\tau^b&\mbox{for any }a,b\in\bZ,\\
\mbox{in $\Aut(G)$:}\hspace{1.1cm}\phi_s\theta^c & = \theta^{cs}\phi_s &\mbox{for any }c\in\bZ,s\in U(e),\\
\mbox{in $\Hol(G)$:}\hspace{0.75cm}\pi\rho(x) &= \rho(\pi(x))\pi&\mbox{for any }x\in G,\pi\in\Aut(G),
\end{align*}
where $k^{-1}$ is to be interpreted as the multiplicative inverse of $k$ mod $e$ for $b$ negative. Following \cite[(5)]{Byott squarefree}, for $h\in\bZ$ and $i\in\bN_{\geq0}$, let us define
\[ S(h,i) = \sum_{\ell=0}^{i-1} h^\ell = 1 + h + \cdots + h^{i-1},\]
with the empty sum $S(h,0)$ representing zero. Then, as in \cite[(6)]{Byott squarefree}, we have
\[ (\sigma^a\tau)^b = \sigma^{aS(k,b)}\tau^b\mbox{ for any }a\in\bZ,b\in\bN_{\geq0}.\]
Using this, a simple induction on $c$ shows that
\[(\theta^c\phi_s)(\sigma^a\tau^b) = \sigma^{as + czS(k,b)}\tau^b\mbox{ for any }a,c\in\bZ,b\in\bN_{\geq0},s\in U(e).\]
We shall frequently use above relations without referring to them explicitly. In what follows, consider an arbitrary element
\[\Phi= \rho(\sigma^{a}\tau^{b})\theta^{c}\phi_{s}\hspace{1em}(a,c\in\bZ,b\in\bN_{\geq0},s\in U(e)) \]
of $\Hol(G)$. The next lemma is essentially \cite[Lemma 4.2]{Byott squarefree}.

\begin{lem}\label{Phi i lem}For any $i\in\bN_{\geq0}$, we have
\[ \Phi^i =  \rho(\sigma^{A_{\Phi}(i)}\tau^{bi})\theta^{cS(s,i)}\phi_{s^i},\]
where we define
\[ A_\Phi(i) = aS(sk^b,i) + czS(k,b)\sum_{\ell=1}^{i-1}S(s,\ell)k^{b\ell}.\]
%In particular, we have $\Phi^i = \Id_G$ if and only if
%\begin{align*}A_\Phi(i)&\equiv0\mmod{e}\\
% bi&\equiv0\mmod{d}\\
% cS(s,i)&\equiv0\mmod{g}\\
% s^i&\equiv1\mmod{e}
%\end{align*}
%are all satisfied.
\end{lem}
\begin{proof}We shall use induction. The case $i=0$ is clear. Suppose now that the claim holds for $i$. Then, we have
\begin{align*}\Phi^{i+1} & = \rho(\sigma^{A_{\Phi}(i)}\tau^{bi})\theta^{cS(s,i)}\phi_{s^i}\cdot\rho(\sigma^a\tau^b)\theta^c\phi_s\\
& = \rho(\sigma^{A_{\Phi}(i)}\tau^{bi}\cdot\sigma^{as^i+cS(s,i)zS(k,b)}\tau^b)\cdot\theta^{cS(s,i)}\phi_{s^i}\cdot\theta^c\phi_s\\
& = \rho(\sigma^{A_{\Phi}(i) + (as^i + czS(k,b)S(s,i))k^{bi}}\tau^{b(i+1)})\cdot\theta^{c(S(s,i)+s^i)}\phi_{s^{i+1}}.
\end{align*}
Plainly, we have $S(s,i)+s^i = S(s,i+1)$, and the exponent of $\sigma$ simplifies to
\[ a(S(sk^b,i)+(sk^b)^i) + czS(k,b)\left(\sum_{\ell=1}^{i-1}S(s,\ell)k^{b\ell} + S(s,i)k^{bi} \right),\]
which is clearly equal to $A_{\Phi}(i+1)$. This proves the claim.
\end{proof}

%Lemma~\ref{Phi i lem} in particular implies the following corollaries.

\begin{lem}\label{order e lem} The element $\Phi$ has order $e$ and satisfies
\begin{equation}\label{order e}\#\{\Phi^i(1_G):i\in\bN_{\geq0}\} = e\end{equation}
precisely when
\[ b\equiv0\mmod{d},\,\ s\equiv1\mmod{e},\,\ \gcd(a,e)=1.\]
\end{lem}
\begin{proof}We may assume $b\equiv0$ (mod $d$), for otherwise $\Phi^e\neq\Id_G$ by Lemma~\ref{Phi i lem} because $\gcd(d,e)=1$. In this case, for any $i\in\bN_{\geq0}$, we have
\[ \Phi^i(1_G) = \sigma^{-A_{\Phi}(i)}\mbox{ and }A_{\Phi}(i) \equiv aS(s,i)\mmod{e}.\]
Suppose that $\Phi$ has order $e$ and (\ref{order e}) holds. For any $i,i'\in\bN_{\geq0}$, we then have
\begin{equation}\label{iff} aS(s,i) \equiv aS(s,i')\mmod{e}\iff i \equiv i' \mmod{e}.\end{equation}
This clearly yields $\gcd(a,e)=1$. The above also implies that $s\equiv1$ (mod $e$). To see why, consider the set
\[ \mathcal{P} = \{p:p\mbox{ is a prime divisor of $e$ such that $i_p\neq1$}\},\]
where $i_p$ denotes the multiplicative order of $s$ mod $p$. Also, let $i$ denote the multiplicative order of $s$ mod $e$, and note that $S(s,i)\equiv0$ (mod $p$) for $p\in\mathcal{P}$. Since $e$ is squarefree, we then deduce that
\[ S(s, e^*i) \equiv e^*S(s,i) \equiv 0\mmod{e},\mbox{ where }e^* = e\prod_{p\in\mathcal{P}}p^{-1}.\]
It follows from (\ref{iff}) that
\[ e^*i\equiv0\mmod{e}\mbox{ and so }\prod_{p\in\mathcal{P}}p\mbox{ divides }i.\]
But we also know that
\[ i = \lcm_{p\in\mathcal{P}}i_p \mbox{ and so }i\mbox{ divides }\prod_{p\in\mathcal{P}}(p-1).\]
This leads to a contradiction unless $\mathcal{P}$ is empty. It follows that $i_p=1$ for all prime divisors $p$ of $e$. Since $e$ is squarefree, this implies that $s\equiv1$ (mod $e$), as claimed. The converse is easily verified. %using Lemma~\ref{Phi i lem}.
\end{proof}

\begin{lem}\label{order d lem}Assume that $b$ is coprime to $d$. Then, the element $\Phi$ has order $d$ precisely when
\[ A_{\Phi}(d)\equiv0\mmod{e},\,\ cS(s,d)\equiv0\mmod{g},\,\ s^d\equiv1\mmod{e}.\]
\end{lem}
\begin{proof}By Lemma~\ref{Phi i lem}, the stated conditions are equivalent to $\Phi^d = \mbox{Id}_G$. In this case, the order of $\Phi$ is exactly $d$ because for any $i\in\bN_{\geq0}$, the exponent of $\tau$ in the expression for $\Phi^i$ in Lemma~\ref{Phi i lem} is $bi$, and $b$ is coprime to $d$.\end{proof}

\subsection{Congruence criteria}

In this subsection, let us fix two elements
\begin{align*}\Phi_1 &= \rho(\sigma^{a_1}\tau^{b_1})\theta^{c_1}\phi_{s_1}\hspace{1em}(a_1,c_1\in\bZ,b_1\in\bN_{\geq0},s_1\in U(e)) \\
\Phi_2 &= \rho(\sigma^{a_2}\tau^{b_2})\theta^{c_2}\phi_{s_2}\hspace{1em}(a_2,c_2\in\bZ,b_2\in\bN_{\geq0},s_2\in U(e))  
\end{align*}
of $\Hol(G)$. In order to compute $\HH_0(G)$, we shall determine exactly when
\begin{equation}\label{conditions}
\begin{cases}
\mbox{$\Phi_1$ has order $e$}\\
\mbox{$\Phi_2$ has order $d$}\\
\mbox{$\Phi_2\Phi_1\Phi_2^{-1} = \Phi_1^k$}\\
\mbox{$\langle\Phi_1,\Phi_2\rangle$ is regular}\\
\mbox{$\langle\Phi_1,\Phi_2\rangle$ is normal in $\Hol(G)$}
\end{cases}
\end{equation}
are all satisfied. %Note that the first three conditions are equivalent to the existence of an injective homomorphism
%\[ \beta: G\longrightarrow \Hol(G)\mbox{ such that }\begin{cases}\beta(\sigma) = \Phi_1,\\\beta(\tau) = \Phi_2.\end{cases}\]
By Lemma~\ref{order e lem}, we may and shall assume that
\[\label{b1 s1 cong}b_1\equiv0\mmod{d}\mbox{ and } s_1\equiv1\mmod{e}.\]
Then, for any $i,j\in\bN_{\geq0}$, using Lemma~\ref{Phi i lem}, we compute that
\begin{align}\label{ij shape}
\Phi_1^i\Phi_2^j & = (\rho(\sigma^{a_1})\theta^{c_1})^i\cdot(\rho(\sigma^{a_2}\tau^{b_2})\theta^{c_2}\phi_{s_2})^j\\\notag
& = \rho(\sigma^{a_1i})\theta^{c_1i}\cdot\rho(\sigma^{A_{\Phi_2}(j)}\tau^{b_2j})\theta^{c_2S(s_2,j)}\phi_{s_2^j}\\\notag
& = \rho(\sigma^{a_1i}\cdot\sigma^{A_{\Phi_2}(j)+c_1izS(k,b_2j)}\tau^{b_2j})\cdot\theta^{c_1i+c_2S(s_2,j)}\phi_{s_2^j}\\\notag
& = \rho(\sigma^{(a_1+c_1zS(k,b_2j))i + A_{\Phi_2}(j)}\tau^{b_2j})\cdot\theta^{c_1i+c_2S(s_2,j)}\phi_{s_2^j}.
\end{align}
This expression shall be used repeatedly in the subsequent calculations.

\begin{lem}\label{k lem}We have $\Phi_2\Phi_1\Phi_2^{-1} = \Phi_1^k$ precisely when
\begin{align*}
a_1(s_2k^{b_2-1}-1) &\equiv c_1zS(k,b_2)\mmod{e},\\
c_1(s_2-k)&\equiv0\mmod{g}.
\end{align*}
\end{lem}
\begin{proof}On the one hand, we have
\begin{align*}
\Phi_2\Phi_1& = \rho(\sigma^{a_2}\tau^{b_2})\theta^{c_2}\phi_{s_2}\cdot \rho(\sigma^{a_1})\theta^{c_1}\\
& = \rho(\sigma^{a_2}\tau^{b_2}\cdot\sigma^{a_1s_2})\cdot\theta^{c_2}\phi_{s_2}\cdot\theta^{c_1}\\
& = \rho(\sigma^{a_2+a_1s_2k^{b_2}}\tau^{b_2})\cdot\theta^{c_2+c_1s_2}\phi_{s_2}.
\end{align*}
On the other hand, from (\ref{ij shape}) and the fact that $A_{\Phi_2}(1) = a_2$, we have
\[ \Phi_1^k\Phi_2 = \rho(\sigma^{(a_1+c_1zS(k,b_2))k+a_2}\tau^{b_2})\cdot\theta^{c_1k+c_2}\phi_{s_2}.\]
The claim then follows by equating the exponents of $\sigma$ and $\theta$.
\end{proof}

\begin{lem}\label{regularity lem}Assume that $\Phi_2$ has order $d$. Then, we have
\begin{equation}\label{Phi1Phi2}\{(\Phi_1^i\Phi_2^j)(1_G) : i,j\in\bN_{\geq0}\} = G\end{equation}
precisely when
\[ \gcd(b_2,d)=1\mbox{ and }\gcd(a_1 + c_1zS(k,b_2j),e)=1\mbox{ for all }j\in\bN_{\geq0}.\]
\end{lem}
\begin{proof}For any $i,j\in\bN_{\geq0}$, from (\ref{ij shape}) we know that
\[ (\Phi_1^i\Phi_2^j)(1_G) = (\sigma^{(a_1+c_1zS(k,b_2j))i+A_{\Phi_2}(j)}\tau^{b_2j})^{-1}.\]
Thus, for (\ref{Phi1Phi2}) to hold, necessarily $\gcd(b_2,d)=1$. In this case, since $\Phi_2$ has order $d$, the equality (\ref{Phi1Phi2}) holds if and only if for each $j\in\bN_{\geq0}$, the set
\[ \{(a_1+c_1zS(k,b_2j))i + A_{\Phi_2}(j): i \in\bN_{\geq0}\}\]
runs over all residue classes mod $e$, which is equivalent to the coefficient of $i$ being invertible mod $e$. This proves the claim.
\end{proof}

\begin{lem}\label{normality lem}Assume that $b_2$ is coprime to $d$, that $\Phi_2$ has order $d$, and that
\begin{equation}\label{group shape}\langle\Phi_1,\Phi_2\rangle = \{\Phi_1^i \Phi_2^j:i,j\in\bN_{\geq0}\}.\end{equation}
Then, we have $\langle\Phi_1,\Phi_2\rangle$ is normal in $\Hol(G)$ precisely when the conditions
\begin{enumerate}[(1)]
\item There exists $i\in\bN_{\geq0}$ such that
\begin{align*}\hspace{5mm}
a_1(k-i)&\equiv c_1z\mmod{e},\\
c_1(i-1)&\equiv0\mmod{g}.
\end{align*}
\item There exists $i\in\bN_{\geq0}$ such that
\begin{align*}\hspace{5mm}
 (a_1+c_1zS(k,b_2))i&\equiv1-s_2k^{b_2}\mmod{e},\\
c_1i&\equiv0\mmod{g}.
\end{align*}
\item There exists $i\in\bN_{\geq0}$ such that
\begin{align*}\hspace{5mm}
(a_1+c_1zS(k,b_2))i&\equiv a_2(k-1)-c_2zk^{b_2}\mmod{e},\\
c_1i&\equiv0\mmod{g}.
\end{align*}
\item There exists $i\in\bN_{\geq0}$ such that
\begin{align*}\hspace{5mm}
 (a_1+c_1zS(k,b_2))i&\equiv zS(k,b_2)\mmod{e},\\
c_1i&\equiv 1- s_2\mmod{g}.
\end{align*}
\item For every $s\in U(e)$, there exists $i\in\bN_{\geq0}$ such that
\begin{align*}\hspace{5mm}
(a_1+c_1zS(k,b_2))i&\equiv a_2(s-1) \mmod{e},\\
c_1i&\equiv c_2(s-1)\mmod{g}.
\end{align*}
\end{enumerate}
are all satisfied. Here, the $i\in\bN_{\geq0}$ are independent of each other.
\end{lem}
\begin{proof}By (\ref{group shape}), the subgroup $\langle\Phi_1,\Phi_2\rangle$ is normal in $\Hol(G)$ if and only if for each $\Psi$ in some set of generators of $\Hol(G)$, we have
\begin{align*}
\Psi\Phi_1\Psi^{-1} & = \Phi_1^{i_{\Psi,1}}\Phi_2^{j_{\Psi,1}}\mbox{ for some }i_{\Psi,1},j_{\Psi,1}\in\bN_{\geq0},\\
\Psi\Phi_2\Psi^{-1} & = \Phi_1^{i_{\Psi,2}}\Phi_2^{j_{\Psi,2}}\mbox{ for some }i_{\Psi,2},j_{\Psi,2}\in\bN_{\geq0}.
\end{align*}
Now, clearly $\Hol(G)$ is generated by $\rho(\sigma),\rho(\tau),\theta$, and $\phi_s$ for $s\in U(e)$. Since
\begin{align*}
\rho(\sigma)(\rho(\sigma^{a_1})\theta^{c_1})\rho(\sigma)^{-1} & = \rho(\sigma^{a_1})\theta^{c_1},\\
\theta(\rho(\sigma^{a_1})\theta^{c_1})\theta^{-1} & = \rho(\sigma^{a_1})\theta^{c_1},\\
\phi_s(\rho(\sigma^{a_1})\theta^{c_1})\phi_s^{-1} & = (\rho(\sigma^{a_1})\theta^{c_1})^s,
\end{align*}
these three equations impose no condition. Next, we compute that
\begin{align*}
\rho(\tau)(\rho(\sigma^{a_1})\theta^{c_1})\rho(\tau)^{-1} &= \rho(\tau\sigma^{a_1}\cdot(\sigma^{c_1z}\tau)^{-1})\cdot\theta^{c_1}\\&= \rho(\sigma^{a_1k-c_1z})\cdot\theta^{c_1},\end{align*}
as well as that
\begin{align*}
\rho(\sigma)(\rho(\sigma^{a_2}\tau^{b_2})\theta^{c_2}\phi_{s_2})\rho(\sigma)^{-1}
& = \rho(\sigma^{a_2+1}\tau^{b_2}\cdot(\sigma^{s_2})^{-1})\cdot\theta^{c_2}\phi_{s_2}\\
& = \rho(\sigma^{a_2+1 - s_2k^{b_2}}\tau^{b_2})\cdot\theta^{c_2}\phi_{s_2}\\
\rho(\tau)(\rho(\sigma^{a_2}\tau^{b_2})\theta^{c_2}\phi_{s_2})\rho(\tau)^{-1}
& = \rho(\sigma^{a_2k}\tau^{b_2+1}\cdot(\sigma^{c_2z}\tau)^{-1})\cdot\theta^{c_2}\phi_{s_2}\\
& = \rho(\sigma^{a_2k - c_2zk^{b_2}}\tau^{b_2})\cdot\theta^{c_2}\phi_{s_2}\\
\theta(\rho(\sigma^{a_2}\tau^{b_2})\theta^{c_2}\phi_{s_2})\theta^{-1}
& = \rho(\sigma^{a_2+zS(k,b_2)}\tau^{b_2})\cdot\theta^{c_2+1}\phi_{s_2}\theta^{-1}\\
& = \rho(\sigma^{a_2+zS(k,b_2)}\tau^{b_2})\cdot\theta^{c_2+1-s_2}\phi_{s_2}\\
\phi_s(\rho(\sigma^{a_2}\tau^{b_2})\theta^{c_2}\phi_{s_2})\phi_s^{-1}
& = \rho(\sigma^{a_2s}\tau^{b_2})\cdot\phi_s\theta^{c_2}\phi_{s_2}\phi_s^{-1}\\
& =  \rho(\sigma^{a_2s}\tau^{b_2})\cdot\theta^{c_2s}\phi_{s_2}.
\end{align*}
Since $\gcd(b_2,d)=1$, by comparing the exponents of $\tau$, we see from (\ref{ij shape}) that
\[ j_{\rho(\tau),1} \equiv0\mmod{d}\mbox{ and }j_{\rho(\sigma),2},j_{\rho(\tau),2},j_{\theta,2},j_{\phi_s,2}\equiv1\mmod{d}\]
necessarily. Since $\Phi_2$ has order $d$, we may assume that these are equalities in $\bZ$. Notice that $A_{\Phi_2}(0) = 0$ and $A_{\Phi_2}(1) = a_2$. Hence, the claim now follows by comparing the exponents of $\sigma$ and $\theta$ in the above with those in (\ref{ij shape}).
\end{proof}

Let us make a useful observation. Given a prime divisor $q$ of $g$, recall that $k\not\equiv1$ (mod $q$) and that $d_q$ denotes the multiplicative order of $k$ mod $q$, so in particular $S(k,d_q)\equiv0$ (mod $q$). For $i,i'\in\bN_{\geq0}$, we then have the relation
\begin{equation}\label{relation} k^{i}S(k,i') \equiv -S(k,i)\mmod{q}\mbox{ when }i+i'\equiv0\mmod{d_q}.\end{equation}
%For any $n\in\bN_{\geq0}$ such that $d_p\nmid n$, we also have
%\begin{align*}
%i\equiv i'\mmod{d_q} &\implies (k^n)^i-1\equiv (k^n)^{i'}-1\mmod{q}\\
%&\implies S(k^n,i)\equiv S(k^n,i')\mmod{q},
%\end{align*}
%which specializes to
%\begin{equation}\label{relation'}S(k^n,d) \equiv S(k^n,d_q) \equiv S(k^n,0) \equiv 0\mmod{q}.\end{equation}
%Then, for any $i,i',n\in\bN_{\geq0}$ with $d_q\nmid n$, we have $k^n\neq1$ (mod $q$) also, and so
%\begin{align*}
%i\equiv i'\mmod{d_q} &\implies (k^{n})^i-1\equiv (k^{n})^{i'}-1\mmod{q}\\
%&\implies S(k^n,i)\equiv S(k^n,i')\mmod{q} 
%\end{align*}
%Hence, for any $j\in\bZ$, not necessarily non-negative, we may define
%\[ S_q(k^n, j) = S(k^n, j + md_q)\hspace{-2.5mm}\mod{q},\]
%where $m$ is any integer such that $j + md_q\geq0$. Note that then
%\[ S_q(k^n,d) = S_q(k^n,d_q) = 0\hspace{-2.5mm}\mod{q},\]
%as we also have the useful relation
We are now ready to give criteria for the conditions in (\ref{conditions}) to all hold.

\begin{prop}\label{criteria prop}The five conditions in $(\ref{conditions})$ simultaneously hold precisely when the conditions
\begin{enumerate}[(a)]
\item We have $\gcd(a_1,e)=1$ and $\gcd(b_2,d)=1$.
\item We have $s_2\equiv1\pmod{z}$ and $a_2\equiv0\pmod{z}$.
%\item For each prime $q\mid z$, we have 
%\[\hspace{5mm}s_2\equiv1\mmod{q}.\]
\item For each prime $q\mid g$ and $q\mid c_1$, we have
\begin{align*}\hspace{5mm}
s_2&\equiv1\mmod{q},\\
b_2&\equiv1\mmod{d_q},\\
c_2&\equiv0\mmod{q}.
\end{align*}
\item For each prime $q\mid g$ and $q\nmid c_1$, we have
\begin{align*}\hspace{5mm}
s_2&\equiv k\mmod{q},\\
b_2&\equiv-1\mmod{d_q},\\
a_1(k-1)&\equiv c_1z\mmod{q},\\
a_2(k-1)&\equiv c_2zk^{b_2}\mmod{q}.
\end{align*}
\end{enumerate}
are all satisfied. 
\end{prop}
\begin{proof}First, suppose that all of the conditions in (\ref{conditions}) hold.
\begin{enumerate}[(a)]
\item By Lemmas~\ref{order e lem} and~\ref{regularity lem}, we have $\gcd(a_1,e)=1$ and $\gcd(b_2,d)=1$.
\item Since $k\equiv1$ (mod $z$), we have have $s_2\equiv1$ (mod $z$) by Lemma~\ref{k lem}. From Lemma~\ref{order d lem}, we then see that
\[\hspace{5mm}a_2d\equiv A_{\Phi_2}(d) \equiv0\mmod{z}.\]
This yields $a_2\equiv0$ (mod $z$) because $d$ is coprime to $e$ and hence to $z$.
\item For each prime $q\mid g$ and $q\mid c_1$, we have $s_2\equiv1$ (mod $q$) by Lemma~\ref{normality lem}(4). From Lemma~\ref{order d lem}, we then see that
\[\hspace{5mm}c_2d \equiv  c_2 S(s_2,d) \equiv 0\mmod{q}.\]
This yields $c_2\equiv0$ (mod $q$) because $d$ is coprime to $e$ and hence to $g$. By Lemma~\ref{k lem}, we also have 
\[ \hspace{5mm} k^{b_2-1}\equiv s_2k^{b_2-1}\equiv 1\mmod{q},\]
which implies that $b_2\equiv1$ (mod $d_q$).
\item For each prime $q\mid g$ and $q\nmid c_1$, we have $s_2\equiv k$ (mod $q$) by Lemma~\ref{k lem}. From Lemma~\ref{normality lem}(2), we then deduce that
\[\hspace{5mm} k^{b_2+1} \equiv s_2k^{b_2}\equiv1\mmod{q},\]
which implies that $b_2\equiv-1$ (mod $d_q$). The other two congruences follow directly from Lemma~\ref{normality lem}(1),(3).
\end{enumerate}
We have thus shown that (a) through (d) are all satisfied.

\vspace{1.5mm}

Conversely, suppose that (a) through (d) are all satisfied. Then, clearly $\Phi_1$ has order $e$ by Lemma~\ref{order e lem}. Hence, it remains to verify that the conditions in Lemmas~\ref{order d lem},~\ref{k lem},~\ref{regularity lem}, and~\ref{normality lem} hold. Since $e$ is squarefree, it suffices to check them mod $z$ as well as mod $q$ for each prime divisor $q$ of $g$.
\begin{enumerate}[(i)]
\item Lemma~\ref{order d lem}: It is obvious that
\[ \hspace{5mm} A_{\Phi_2}(d) \equiv 0\mmod{z}\mbox{ and } s_2^d\equiv1\mmod{e}.\]
Since $S(k,d)\equiv0$ (mod $q$), we also clearly have
\begin{align*}\hspace{5mm}
A_{\Phi_2}(d)& \equiv 0 \mmod{q}\mbox{ for }q\mid c_1,\\
c_2S(s_2,d) &\equiv0 \mmod{q} \mbox{ for both $q\mid c_1$ and $q\nmid c_1$}.\end{align*}
For $q\nmid c_1$, using (\ref{relation}), for any $j\in\bN_{\geq0}$ we compute that
\begin{align*}\hspace{5mm}
A_{\Phi_2}(j) & \equiv a_2S(k^{b_2+1},j)+ a_2k(k-1)S(k,b_2)\sum_{\ell=1}^{j-1}S(k,\ell)k^{b_2\ell}\mmod{q}\\
& \equiv a_2S(1,j) - a_2S(k,1)\sum_{\ell=1}^{j-1}(k^\ell-1)k^{b_2\ell}\mmod{q}\\
& \equiv a_2j - a_2\sum_{\ell=1}^{j-1}(1-k^{b_2\ell})\mmod{q}\\
& \equiv a_2j - a_2((j-1) - (S(k^{b_2},j)-1))\mmod{q}\\
& \equiv a_2S(k^{b_2},j) \mmod{q}.
\end{align*}
Since $b_2$ is coprime to $d$, we have $k^{b_2}\not\equiv1$ (mod $d_q$), and the above yields
\[\hspace{5mm} A_{\Phi_2}(d)\equiv a_2S(k^{b_2},d)\equiv 0\mmod{q}.\]
We have thus shown that $\Phi_2$ has order $d$.
\item Lemma~\ref{k lem}: The conditions there clearly hold so $\Phi_2\Phi_1\Phi_2^{-1}=\Phi_1^k$.
\item Lemma~\ref{regularity lem}: For any $j\in\bN_{\geq0}$, observe that
\[\hspace{5mm}
a_1 + c_1zS(k,b_2j) \equiv \begin{cases}
a_1 & \mbox{(mod $z$), (mod $q$) for $q\mid c_1$},\\
a_1k^{b_2j} &\mbox{(mod $q$) for $q\nmid c_1$}.
\end{cases}\]
Since $a_1$ is coprime to $e$, it follows that the above expression on the left is coprime to $e$. We have thus shown that $\langle\Phi_1,\Phi_2\rangle$ is regular.
\item Lemma~\ref{normality lem}: By (a) and (iii) above, there exist $\widetilde{a_1},\widetilde{u}\in\bZ$ such that
\[\hspace{5mm} a_1\widetilde{a_1} \equiv1\mmod{e}\mbox{ and }(a_1+c_1zS(k,b_2))\widetilde{u}\equiv1\mmod{e}.\]
To exhibit the existence of $i\in \bN_{\geq0}$ claimed in Lemma~\ref{normality lem}, simply take
\[\hspace{5mm} i \equiv \begin{cases}
\widetilde{a_1}(a_1k-c_1z) &\mbox{(mod $e$) in part (1)},\\
\widetilde{u}(1-s_2k^{b_2}) &\mbox{(mod $e$) in part (2)},\\
\widetilde{u}(a_2(k-1) - c_2zk^{b_2})&\mbox{(mod $e$) in part (3)},\\
\widetilde{u}zS(k,b_2)&\mbox{(mod $e$) in part (4)},\\
\widetilde{u}a_2(s-1)&\mbox{(mod $e$) in part (5)}.
\end{cases}\]
In each part of Lemma~\ref{normality lem}, the first congruence then clearly holds, and for $q\mid c_1$, the second congruence also holds. For $q\nmid c_1$, observe that
\begin{align*}\hspace{5mm}
\widetilde{a_1}(a_1k-c_1z)& \equiv k - (k-1) \equiv 1\mmod{q},\\
\widetilde{u}(1-s_2k^{b_2})& \equiv \widetilde{u}(1-k^{b_2+1}) \equiv 0\mmod{q},\\
\widetilde{u}(a_2(k-1) - c_2zk^{b_2}) & \equiv \widetilde{u}(c_2zk^{b_2} - c_2zk^{b_2}) \equiv 0\mmod{q},\end{align*}
and so the second congruence holds in parts (1), (2), and (3). Also, note that by (iii) above, we have $\widetilde{u}\equiv \widetilde{a_1}k$ (mod $q$). We then compute that
\begin{align*}
c_1(\widetilde{u}zS(k,b_2)) & \equiv k(k-1)S(k,b_2) \mmod{q}\\
& \equiv k(k^{b_2}-1) \mmod{q}\\
& \equiv k^{b_2+1}-k \mmod{q}\\
& \equiv 1-s_2 \mmod{q},
\end{align*}
which verifies the second congruence in part (4), and similarly that
\begin{align*}
z\cdot c_1(\widetilde{u}a_2(s-1))& \equiv a_2k(k-1) (s-1)\mmod{q}\\
&\equiv z\cdot c_2(s-1)\mmod{q},
\end{align*}
which verifies the second congruence in part (5) since $z$ is coprime to $g$. We have thus shown that $\langle\Phi_1,\Phi_2\rangle$ is normal in $\Hol(G)$.
\end{enumerate}
This completes the proof of the proposition.
\end{proof}

%Proposition~\ref{criteria prop} shall allow us to count the size of $\HH_0(G)$ which is equal to the order of $T(G)$. 
In order to compute the actual group structure of $T(G)$, we shall also need to understand the map $\xi_{\langle\Phi_1,\Phi_2\rangle}$ defined in (\ref{xi}).

\begin{prop}\label{xi prop}Assume that the conditions $(a)$ to $(d)$ in Proposition~$\ref{criteria prop}$ are all satisfied. Then, for any $i,j\in\bN_{\geq0}$, we have
\[ (\Phi_1^i\Phi_2^j)(1_G) = \sigma^{-\fa(i,j)}\tau^{-b_2j},\]
where we define
\begin{equation}\label{fa} \fa(i,j) = \widetilde{k}^{b_2j}((a_1+c_1zS(k,b_2j))i + A_{\Phi_2}(j)),\end{equation}
and $\widetilde{k}\in\bZ$ is such that $k\widetilde{k}\equiv1\pmod{e}$. Moreover, we have
\[ \fa(i,j) \equiv \begin{cases}
a_1i&\hspace{-2mm}\mmod{z},\\
\widetilde{k}^{j}(a_1i + a_2S(k,j))&\hspace{-2mm}\mmod{q}\mbox{ for a prime }q\mid g\mbox{ and }q\mid c_1,\\
a_1i + a_2kS(k,j)&\hspace{-2mm}\mmod{q}\mbox{ for a prime }q\mid g\mbox{ and }q\nmid c_1.\end{cases}\]
\end{prop}
\begin{proof}The first claim follows from (\ref{ij shape}). The congruence mod $z$ is also clear since $k\equiv1$ (mod $z$) and $A_{\Phi_2}(j)\equiv0$ (mod $z$). For a prime $q\mid g$, recall that
\[ b_2\equiv \begin{cases}1&\hspace{-2mm}\mmod{d_q}\mbox{ for }q\mid c_1,\\
-1&\hspace{-2mm}\mmod{d_q}\mbox{ for }q\nmid c_1.\end{cases}\]
For $q\mid c_1$, we have $A_{\Phi_2}(j)\equiv a_2S(k,j)$ (mod $q$), and so the claim is clear. As for $q\nmid c_1$, observe that
\begin{align*}a_1+c_1zS(k,b_2j) &\equiv a_1k^{b_2j} \mmod{q}\\ A_{\Phi_2}(j)&\equiv a_2S(k^{b_2},j)\mmod{q}
\end{align*}
by (iii) and (i) in the proof of Proposition~\ref{criteria prop}. Since $k^{b_2}\equiv \widetilde{k}$ (mod $q$) in this case, we deduce that
\begin{align*}
\fa(i,j) & \equiv k^j(a_1\widetilde{k}^{j}i + a_2S(\widetilde{k},j))\mmod{q}\\
& \equiv a_1i + a_2k^jS(\widetilde{k},j)\mmod{q}\\
& \equiv a_1i + a_2kS(k,j)\mmod{q},
\end{align*}
as claimed. This proves the proposition.
\end{proof}

\subsection{Proof of Theorem~\ref{main thm}} Let us first make a definition.

\begin{definition}Given a tuple $\fm=(a_1,c_1,a_2,b_2,c_2,s_2)$ of integers, we shall say that $\fm$ is \emph{admissible} if the conditions (a) to (d) in Proposition~\ref{criteria prop} are all satisfied. The admissibility of $\fm$ depends only on its class mod $\mathbb{M}$, where 
\[\mathbb{M} = e\bZ\times g\bZ\times e\bZ\times d\bZ\times g\bZ\times e\bZ.\]
In the case that $\fm$ is admissible, define
\[ \Phi_{\fm,1} = \rho(\sigma^{a_1})\theta^{c_1}\mbox{ and }
\Phi_{\fm,2} = \rho(\sigma^{a_2}\tau^{b_2})\theta^{c_2}\phi_{s_2},\]
which satisfy (\ref{conditions}) by Proposition~\ref{criteria prop}. Further, define
\[ \pi_\fm\in\Perm(G);\hspace{1em} \pi_\fm(\sigma^i\tau^j) = \phi_{-1}(\Phi_{\fm,1}^i\Phi_{\fm,2}^j(1_G)).\]
An explicit description of $\pi_\fm$ was given in Proposition~\ref{xi prop} and the $\phi_{-1}$ above is only to remove the negative sign in the exponent of $\sigma$ for convenience. 
\end{definition}

Recall that $T(G)$ has the same size as $\HH_0(G)$. By (\ref{HH0}), we plainly have
\begin{align*}
\#\HH_0(G) & = \frac{1}{|\Aut(G)|}\cdot\#\left\{\beta\in\mbox{Hom}(G,\Hol(G)):\begin{array}{c}
\beta(G)\mbox{ is regular and}\\\mbox{is normal in }\Hol(G)
\end{array}\right\}\\
& = \frac{1}{|\Aut(G)|}\cdot\#\{(\Phi_1,\Phi_2)\in\Hol(G)^2:\mbox{the conditions in (\ref{conditions}) hold}\}.
%& = \frac{1}{|\Aut(G)|}\#\{\mbox{admissible tuples }(a_1,c_1,a_2,b_2,c_2,s_2)\mmod{\mathbb{M}}\} 
\end{align*}
From Lemma~\ref{Aut lem} and Proposition~\ref{criteria prop}, we then deduce that
\[ |T(G)| = \frac{1}{g\varphi(e)}\cdot\#\{\mbox{admissible tuples $(a_1,c_1,a_2,b_2,c_2,s_2)$ mod $\mathbb{M}$}\}.\]
We shall now compute this number.

\begin{proof}[Proof for the order] Consider a pair $(f_1,f_2)\in\bN^2$ with $g=f_1f_2$. Below, let $q$ be an arbitrary prime divisor of $g$, and we shall count the admissible tuples 
\[\fm = (a_1,c_1,a_2,b_2,c_2,s_2)\mod{\mathbb{M}} \]
under the restriction that 
\begin{equation}\label{c1}q\mid c_1\mbox{ if and only if }q\mid f_1.\end{equation}
Since $g$ is squarefree, this is equivalent to
\[ q\nmid c_1 \mbox{ if and only if }q\mid f_2.\]
Since $d$ is squarefree and $d = \lcm_{q\mid g}d_q$, there exists $b_2\in\bZ$ satisfying
\[ b_2\equiv\begin{cases}
1&\hspace{-2mm}\mmod{d_q}\mbox{ for all $q\mid f_1$}\\
-1&\hspace{-2mm}\mmod{d_q}\mbox{ for all $q\mid f_2$}
\end{cases}\]
exactly when $(f_1,f_2)\in\mathfrak{o}(G)$. Hence, for $(f_1,f_2)\notin\mathfrak{o}(G)$, we get no admissible tuple satisfying (\ref{c1}). So let us assume that $(f_1,f_2)\in\mathfrak{o}(G)$. Once we impose (\ref{c1}), in order for $\fm$ to satisfy the conditions in Proposition~\ref{criteria prop}, the classes $s_2$ mod $e$ and  $b_2$ mod $d$ are uniquely determined. Note that mod $z$, we have
\begin{itemize}
\item $\varphi(z)$ choices of $a_1$,
\item $1$ choice of $a_2$.
\end{itemize}
Similarly, mod $q$ for $q\mid f_1$, we have
\begin{itemize}
\item $1$ choice of $c_1$ and $c_2$,
\item $\varphi(q)$ choices of $a_1$,
\item $q$ choices of $a_2$.
\end{itemize}
Finally, mod $q$ for $q\mid f_2$, we have
\begin{itemize}
\item $\varphi(q)$ choices of $c_1$ and $c_1$ determines $a_1$ uniquely,
\item $q$ choices of $c_2$ and $c_2$ determines $a_2$ uniquely,
\end{itemize}
because $k\not\equiv1$ (mod $q$). It follows that mod $\mathbb{M}$, we have 
\[ \varphi(z)\prod_{q\mid f_1}q\varphi(q)\prod_{q\mid f_2}q\varphi(q) = g\varphi(e)\]
admissible tuples satisfying (\ref{c1}). Therefore, in total, we have $\#\mathfrak{o}(G)\cdot g\varphi(e)$ admissible tuples without any restriction, so indeed $|T(G)| = \#\mathfrak{o}(G)$.\end{proof}

A proof of (\ref{HH0'}) was given in \cite[Section 2]{Tsang NHol} and the key is that isomorphic regular subgroups of $\Perm(G)$ are always conjugates. In fact, by the proof of \cite[Lemma 2.1]{Tsang NHol}, if $N_1$ and $N_2$ are isomorphic regular subgroups of $\Perm(G)$, via the isomorphism $\psi:N_1\longrightarrow N_2$ say, then
\[ N_2 = \pi N_1 \pi^{-1} \mbox{ for }\pi = \xi_{N_2}\circ\psi\circ \xi_{N_1}^{-1}.\]
For any admissible tuple $\fm=(a_1,c_1,a_2,b_2,c_2,s_2)$, by taking
\[ \psi_\fm: \lambda(G)\longrightarrow\langle\Phi_{\fm,1},\Phi_{\fm,2}\rangle;\hspace{1em}\begin{cases}\psi_{\fm}(\lambda(\sigma)) = \Phi_{\fm,1}\\\psi_{\fm}(\lambda(\tau)) = \Phi_{\fm,2}\end{cases} \]
to be the natural isomorphism, we see that
\[\langle\Phi_{\fm,1},\Phi_{\fm,2}\rangle = \widetilde{\pi_\fm}\lambda(G)\widetilde{\pi_\fm}^{-1} \mbox{ for }\widetilde{\pi_\fm}(\sigma^i\tau^j) = (\Phi_{\fm,1}^i\Phi_{\fm,2}^j)(1_G).\]
Note that $\pi_\fm = \phi_{-1}\circ\widetilde{\pi_\fm}$ and $\phi_{-1}\in\Aut(G)$. By (\ref{HH0'}), we then have%It follows that
%\[ \HH_0(G) = \{\pi_\fm\lambda(G)\pi_\fm^{-1}:\fm\mbox{ is an admissible tuple}\},\]
%and in partiular
\[ T(G) = \{\pi_\fm \Hol(G): \fm\mbox{ is an admissible tuple}\}.\]
We shall now compute the structure of $T(G)$.

\begin{proof}[Proof for the structure] It suffices to show that the exponent of $T(G)$ divides two because  any group of exponent two is necessarily abelian. Let
\[ \fm = (a_1,c_1,a_2,b_2,c_2,s_2)\]
be an admissible tuple and we need to show that $\pi_\fm^2\in\Hol(G)$. Note that for any $i,j\in\bN_{\geq0}$, by Proposition~\ref{xi prop}, we have
\[ \pi_\fm(\sigma^i\tau^j) = \sigma^{\fa_\fm(i,j)}\tau^{-b_2j}\mbox{ and so }\pi_\fm^2(\sigma^i\tau^j) = \sigma^{\fv_{\fm}(i,j)}\tau^{j},\]
where the exponent of $\tau$ is simply $j$ because $b_2^2\equiv1$ (mod $d$) by admissibility. Here $\fa_\fm(i,j)$ is as in (\ref{fa}) and $ \fv_{\fm}(i,j)$ may be computed explicitly as follows. 

\vspace{1.5mm}

First, it is clear that we have
\[ \fv_{\fm}(i,j) \equiv a_1(a_1i) \equiv a_1^2i\mmod{z}. \]
Next, let $q$ denote an arbitrary prime divisor of $g$. Recall that
 \[ b_2\equiv\begin{cases}
1&\hspace{-2mm}\mmod{d_q}\mbox{ for $q\mid c_1$}\\
-1&\hspace{-2mm}\mmod{d_q}\mbox{ for $q\nmid c_1$}
\end{cases}\mbox{ and }
s_2\equiv\begin{cases}
1&\hspace{-2mm}\mmod{d_q}\mbox{ for $q\mid c_1$}\\
k&\hspace{-2mm}\mmod{d_q}\mbox{ for $q\nmid c_1$}
\end{cases}\]
by admissibility. Without loss of generality, let us assume that $0\leq j \leq d-1$. %We shall consider two  different cases.
\begin{enumerate}[$\bullet$]
\item $q\mid c_1:$ Using Proposition~\ref{xi prop} and (\ref{relation}), we compute that
\begin{align*}
\fv_{\fm}(i,j) &\equiv \widetilde{k}^{d-j}(a_1(\widetilde{k}^j(a_1i+a_2S(k,j))) + a_2S(k,d-j))\mmod{q}\\
&\equiv a_1^2i + a_1a_2S(k,j) + a_2k^jS(k,d-j)\mmod{q}\\
&\equiv a_1^2i + (a_1 - 1)a_2S(k,j)\mmod{q}.
\end{align*}
\item $q\nmid c_1:$ Using Proposition~\ref{xi prop}, we compute that
\begin{align*}
\fv_{\fm}(i,j) & \equiv a_1(a_1i+a_2kS(k,j)) + a_2kS(k,j)\mmod{q}\\
%& \equiv a_1^2i + a_1a_2kS(k,j) + a_2kS(k,j)\\
& \equiv a_1^2i + (a_1+1)a_2kS(k,j)\mmod{q}.
\end{align*}
\end{enumerate}
Since $e=gz$ is squarefree, we then deduce that
\[ \fv_{\fm}(i,j) \equiv a_1^2i + z\widetilde{z}(a_1 - u)a_2s_2S(k,j) \mmod{e},\]
where $\widetilde{z},u\in\bZ$ are such that
\[ z\widetilde{z}\equiv1\mmod{g}\mbox{ and }u \equiv \begin{cases}
1 &\hspace{-2mm}\mmod{q}\mbox{ for $q\mid c_1$},\\ -1 &\hspace{-2mm}\mmod{q}\mbox{ for $q\nmid c_1$}.
\end{cases}\]
 It follows that
\[\pi_{\fm}^2(\sigma^i\tau^j)  = \sigma^{a_1^2i+\widetilde{z}(a_1-u)a_2s_2zS(k,j)}\tau^{j} = (\theta^{\widetilde{z}(a_1-u)a_2s_2}\phi_{a_1^2})(\sigma^i\tau^j).\]
This shows that $\pi_\fm^2\in\Aut(G)$ so $T(G)$ is indeed elementary $2$-abelian.
\end{proof}

\section{Groups of odd prime power order}\label{section2}

Throughout this section, let $G$ denote a finite group of order a power of an odd prime $p$, and of nilpotency class $n_c$. Write $Z(G)$ for its center and let
\[ \gamma_1(G) \geq \gamma_2(G)\geq \gamma_3(G)\geq\cdots\]
be its lower central series, namely
\[ \gamma_1(G) = G\mbox{ and }\gamma_i(G) = [G,\gamma_{i-1}(G)]\mbox{ for }i\geq 2.\]
Note that $\gamma_i(G) = 1$ for $i\geq n_c+1$. Write
\[ \exp(G/Z(G)) = p^r\mbox{ and }\exp(\gamma_3(G)) = p^t,\mbox{ where }r,t\in\bN_{\geq0}.\]
For convenience, let us assume that $r\geq 1$, namely $n_c\geq2$, for otherwise $G$ is abelian and the structure of $T(G)$ is known. Consider the subgroup
\[ T_0(G) = \{n+p^r\bZ : p\nmid n\mbox{ and } n \equiv 1\mbox{ (mod $p^t$)}\}\]
of the cyclic multiplicative group $(\bZ/p^r\bZ)^\times$. Note that
\[ |T_0(G)| = \begin{cases} (p-1)p^{r-1}&\mbox{if $t=0$, that is $n_c=2$},\\
\max\{1,p^{r-t}\}&\mbox{if $t\geq1$, that is $n_c\geq3$}.
\end{cases}\]
Further, for each $n\in\bZ$, define
\[ \pi_n:G\longrightarrow G;\hspace{1em}\pi_n(x) = x^n\]
to be the $n$th power map on $G$. We shall prove:

\begin{thm}\label{main thm'}Assume that $2\leq n_c\leq p-1$. Then, the map
\[ \kappa:T_0(G)\longrightarrow T(G);\hspace{1em}\kappa(n+p^r\bZ) = \pi_n\Hol(G)\]
is a well-defined injective homomorphism. 
\end{thm}

Let us remark that Theorem~\ref{main thm'} is motivated by \cite[Proposition 3.1]{Caranti3}. The latter treats the special case when $n_c=2$.

\vspace{1.5mm}

To prove Theorem~\ref{main thm'}, rather than (\ref{HH0'}) and (\ref{HH0}), we shall use the definition (\ref{NHol def}) directly. For each $n\in\bZ$ coprime to $p$, clearly $\pi_n$ is a bijection that centralizes $\Aut(G)$, whence we have
\begin{equation}\label{pi n}
\pi_n\in\NHol(G) \iff \pi_n\circ\rho(G)\circ\pi_n^{-1}\leq \Hol(G).
\end{equation}
For each $\sigma\in G$, let us further define
\[\zeta_{n,\sigma}\in\Perm(G);\hspace{1em}\zeta_{n,\sigma}= \rho(\sigma^{n})^{-1}\circ\pi_n\circ\rho(\sigma)\circ\pi_n^{-1}.\]
Explicitly, for any $x\in G$, we have
\begin{equation}\label{zeta}\zeta_{n,\sigma}(x) =  (x^{\widetilde{n}}\sigma^{-1})^n\sigma^{n},\end{equation}
where $\widetilde{n}\in\bZ$ is such that $n\widetilde{n}\equiv1$ (mod $\exp(G)$).%We shall need the next lemma.

\begin{lem}\label{zeta lem}Let $n\in\bZ$ be coprime to $p$ and let $\sigma\in G$. Then, we have
\[ \pi_n\circ\rho(\sigma)\circ\pi_n^{-1}\in\Hol(G)\iff \zeta_{n,\sigma}\in\Aut(G).\]
\end{lem}
\begin{proof}This is because $\pi_n\circ\rho(\sigma)\circ\pi_n^{-1}$ sends $1_G$ to $\sigma^{-n}$. 
%\[ \pi_n\circ\rho(\sigma)\circ\pi_n^{-1}\in\Hol(G)\iff \rho(\sigma^n)^{-1}\circ\pi_n\circ\rho(\sigma)\circ\pi_n^{-1}\in\Aut(G).\]
%The element on the right is equal to $\zeta_{n,\sigma}$ and so the claim follows.
\end{proof}

\subsection{Some commutator calculations}

In this subsection, let $x,y\in G$ be two arbitrary elements. In the case that $n_c=2$, as shown in  \cite[(5.3.5)]{Rob book}, for example, for any $n\in\bN$, we have the well-known formula
\begin{equation}\label{c2 formula} x^ny^n = (xy)^n[x,y]^{{n \choose 2}},\mbox{ where }[x,y] = x^{-1}y^{-1}xy,\end{equation}
and using this, it may be verified that (\ref{zeta}) is a homomorphism. This was in fact the key idea to the proof of \cite[Proposition 3.1]{Caranti3}. In general, although (\ref{c2 formula}) might not hold, the difference $(xy)^{-n}x^ny^n$ is still a product of commutators. This difference was first studied by P. Hall  in \cite{Hall}. For our purpose, we shall use the following so-called Hall-Petresco formula for its description.

\begin{lem}\label{Hall lem}For any $n\in\bN$, we have the formula 
\[x^ny^n = (xy)^n c_2(x,y)^{{n \choose 2}}c_3(x,y)^{{n\choose 3}}\cdots c_n(x,y)^{{n\choose n}},\]
where $c_2(x,y) =[x,y][[x,y],y]$ and $c_i(x,y)\in \gamma_i(\langle x,y\rangle)$ for each $i=2,\dots,n$. 
\end{lem}
\begin{proof}See \cite[Theorem 3.5]{Suzuki book} and the discussion after it.
\end{proof}

Our idea is to impose conditions on $n\in\bN$ and the nilpotency class $n_c$, so that we may reduce the formula in Lemma~\ref{Hall lem} to (\ref{c2 formula}), even when $n_c\geq 3$.

\begin{lem}\label{formula lem}Assume that $2\leq n_c\leq p-1$. Then, for any $n\in\bN$ such that
\[ n\equiv\ep\mmod{p^t},\mbox{ where }\ep\in\{0,1\},\]
the formula $(\ref{c2 formula})$ holds.
%\begin{enumerate}[(a)]
%\item For any $n\in\bN$ with $n\equiv0\pmod{p^t}$, we have $x^ny^n = (xy)^nc_2(x,y)^{{n\choose 2}}$.
%\item For any $n\in\bN$ with $n\equiv1\pmod{p^t}$, we have $x^ny^n = (xy)^n[x,y]^{{n\choose 2}}$.
%\end{enumerate}
\end{lem}
\begin{proof}We may assume that $n_c\geq3$, which means that $t\geq1$, whence $p\leq n$. Hence, we have $n_c<p\leq n$, and since $\gamma_i(G)=1$ for $i\geq n_c+1$, the formula in Lemma~\ref{Hall lem} may be rewritten as
\[x^ny^n = (xy)^n c_2(x,y)^{{n \choose 2}}c_3(x,y)^{{n\choose 3}}\cdots c_{n_c}(x,y)^{{n\choose n_c}}.\]
For each $2\leq i\leq n_c$, observe that $p^t$ divides $n-\ep$ and hence divides
\[ {n\choose i} = \frac{n(n-1)\cdots (n-i+1)}{i!},\]
because $i\leq n_c<p$ implies $i!$ has no factor of $p$. Since $\gamma_3(G)$ has exponent $p^t$ by definition, we then deduce that 
\[ c_i(x,y)^{{n\choose i}} = 1_G\mbox{ for all }3\leq i\leq n_c,\mbox{ whence }x^ny^n = (xy)^n c_2(x,y)^{{n\choose 2}}.\]
Put $N={n\choose 2}$, which is divisible by $p^t$. We may then apply the above formula again to the elements $[x,y]$ and $[[x,y],y]$ in $c_2(x,y)$ to obtain
\begin{align*}x^ny^n & = (xy)^nc_2(x,y)^N\\
& = (xy)^n[x,y]^N[[x,y],y]^Nc_2([x,y],[[x,y],y])^{-{N\choose 2}}.\end{align*}
Again, since $\gamma_3(G)$ has exponent $p^t$, we have
\[ [[x,y],y]^N, c_2([x,y],[[x,y],y])^{{N\choose 2}} = 1_G,\mbox{ whence }x^ny^n = (xy)^n [x,y]^{{n\choose 2}}.\]
This proves the claim.
\end{proof}

The next lemma is well-known in the special case $n_c = 2$. Recall that $G$ is said to be \emph{regular} (as a $p$-group) if for every $\sigma,\tau\in G$ and $i\in\mathbb{N}$, we have
\[ \sigma^{p^i}\tau^{p^i} = (\sigma\tau)^{p^i}\mu^{p^i}\mbox{ for some }\mu\in \gamma_2(\langle\sigma,\tau\rangle).\]
It is known, by Lemma~\ref{Hall lem}, that $G$ is regular when $n_c\leq p-1$. We thank the referee for pointing out that Lemma~\ref{extra lem} below is true as long as $G$ is regular and for giving the following simple proof.

\begin{lem}\label{extra lem}Assume that $G$ is regular (as a $p$-group). Then, for any $n\in\bN$ which is a multiple of $p^t$, we have
\begin{enumerate}[(a)]
\item $[x,y]^n \in Z(G)$;
\item $[\sigma,xy]^{n} = [\sigma,x]^{n}[\sigma,y]^{n}$ for all $\sigma\in G$;
\item $[x,y]^{n} = [x,y^n]$.
\end{enumerate}
\end{lem}
\begin{proof}Let $n=mp^t$ with $m\in\bN$. For any $z\in G$, by \cite[Theorem 12.4.3 (1)]{MHall book} and regularity, we have the equivalence
\[ [[x,y]^{p^t},z] =1_G \mbox{ if and only if }[[x,y],z]^{p^t}=1_G.\]
The latter holds because $\gamma_3(G)$ has exponent $p^t$. This yields $[x,y]^{p^t}\in Z(G)$, which clearly implies part (a). For any $\sigma\in G$, we have the identity
\[ [\sigma,xy] = [\sigma,y][\sigma,x][[\sigma,x],y].\]
By regularity and the fact that $\gamma_3(G)$ has exponent $p^t$, we obtain
\begin{equation}\label{eqn}[\sigma,xy]^{p^t}= ([\sigma,y][\sigma,x])^{p^t} [[\sigma,x],y]^{p^t} \mu_1^{p^t} = [\sigma,y]^{p^t}[\sigma,x]^{p^t}\mu_2^{p^t}\mu_1^{p^t},\end{equation}
where $\mu_1$ and $\mu_2$, respectively, belong to the subgroups
\[ \gamma_2(\langle [\sigma,y][\sigma,x],[[\sigma,x],y]\rangle)\mbox{ and } \gamma_2(\langle[\sigma,y],[\sigma,x]\rangle).\]
Both of them lie in $\gamma_3(G)$ and thus $\mu_1^{p^t},\mu_2^{p^t}=1_G$. Since $[\sigma,x]^{p^t},[\sigma,y]^{p^t}\in Z(G)$ by part (a), raising (\ref{eqn}) to the $m$th power gives part (b). Now, by regularity, we also have
\begin{equation}\label{eqn'}[x,y]^{p^t} = (x^{-1}y^{-1}x\cdot y)^{p^t} = (x^{-1}y^{-p^{t}}x)\cdot y^{p^t}\cdot \mu^{p^t} = [x,y^{p^t}]\mu^{p^t},\end{equation}
where $\mu$ belongs to the subgroup
\[ \gamma_2(\langle x^{-1}y^{-1}x,y\rangle) = \gamma_2(\langle[x,y]y^{-1},y\rangle) = \gamma_2(\langle[x,y],y\rangle).\]
This lies in $\gamma_3(G)$ and so $\mu^{p^t}=1_G$. Replacing $y$ by $y^m$ in (\ref{eqn'}) then yields
\[ [x,y^m]^{p^t} = [x,y^n].\]
But by part (b), we know that
\[ [x,y^m]^{p^t} = [x,y]^{p^t}\cdots [x,y]^{p^t} = [x,y]^n\]
and this proves part (c).
\end{proof}

\subsection{Proof of Theorem~\ref{main thm'}}

Assume that $2\leq n_c\leq p-1$. In the following, let $n\in\bN$ be such that $p\nmid n$ and $n\equiv1$ (mod $p^t$). Let $\widetilde{n}\in\bN$ be as after (\ref{zeta}), and we have $\widetilde{n}\equiv1$ (mod $p^t)$. Also, note that $p^t$ divides both ${n\choose 2}$ and ${\widetilde{n}\choose 2}$. 

\vspace{1.5mm}

First, let $\sigma\in G$ and consider the map $\zeta_{n,\sigma}$ defined in (\ref{zeta}). By Lemmas~\ref{formula lem} and~\ref{extra lem}(a), for any $x\in G$, we may rewrite
\[ \zeta_{n,\sigma}(x) = (x^{\widetilde{n}}\sigma^{-1})^n\sigma^n\\
 = x\sigma^{-n}[x^{\widetilde{n}},\sigma^{-1}]^{-{n\choose 2}}\sigma^{n} = x[\sigma^{-1},x^{\widetilde{n}}]^{{n\choose 2}}.\]
Applying Lemmas~\ref{formula lem} and~\ref{extra lem}(a),(b) again, for any $x,y\in G$, we then have
\begin{align*}
\zeta_{n,\sigma}(xy) %&= xy[\sigma^{-1},(xy)^{\widetilde{n}}]^{{n\choose 2}}\\
&= xy[\sigma^{-1},x^{\widetilde{n}}y^{\widetilde{n}}[x,y]^{-{\widetilde{n}\choose 2}}]^{n\choose 2}\\
&= xy[\sigma^{-1},x^{\widetilde{n}}y^{\widetilde{n}}]^{n\choose 2}\\
&= xy[\sigma^{-1},x^{\widetilde{n}}]^{n\choose 2}[\sigma^{-1},y^{\widetilde{n}}]^{n\choose 2}\\
& = x[\sigma^{-1},x^{\widetilde{n}}]^{n\choose 2}y[\sigma^{-1},y^{\widetilde{n}}]^{n\choose 2}\\
& =\zeta_{n,\sigma}(x)\zeta_{n,\sigma}(y)
\end{align*}
and this proves that $\zeta_{n,\sigma}$ is a homomorphism. By Lemma~\ref{zeta lem} and (\ref{pi n}), this in turn implies that $\pi_n\in\NHol(G)$, so that $\pi_n\Hol(G)\in T(G)$. 

\vspace{1.5mm}

Next, observe that since $\pi_n$ sends $1_G$ to $1_G$, we have
\[\pi_n\in\Hol(G)\iff \pi_n\in\Aut(G).\]
From Lemmas~\ref{formula lem} and~\ref{extra lem}(c), we then deduce that
\begin{align*}
\pi_n\in\Hol(G) &\iff [x,y]^{{n\choose 2}} =1_G\mbox{ for all }x,y\in G\\
&\iff [x,y^{n\choose 2}]=1_G\mbox{ for all }x,y\in G\\
&\iff y^{n\choose 2}\in Z(G)\mbox{ for all }y\in G\\
&\iff \exp(G/Z(G))\mbox{ divides }{\textstyle\small{n\choose2}}\\
&\iff n\equiv1\mmod{p^r},
\end{align*}
where the last equivalence holds because $p\nmid n$, which means that the highest power of $p$ dividing $n\choose 2$ equals that of $n-1$. This shows that $\kappa$ is well-defined and injective. Clearly $\kappa$ is a homomorphism and the theorem now follows.

\subsection{Further questions} Theorem~\ref{main thm'} raises two natural questions:
\begin{enumerate}[1.]
\item How often Theorem~\ref{main thm'} may be applied to show that $T(G)$ is not elementary $2$-abelian or not even a $2$-group?
\item Can $T(G)$ be a non-$2$-group or can we generalize Theorem~\ref{main thm'} to the case when $G$ has nilpotency class $n_c\geq p$?
\end{enumerate}
Let us briefly discuss these two questions.

\subsubsection{Applicability}

In the case that $n_c=2$, namely $r\geq1$ and $t=0$, from Theorem~\ref{main thm'} we see that $T(G)$ has a cyclic subgroup of order $(p-1)p^{r-1}$, so in particular
\begin{itemize}
\item $T(G)$ is not elementary $2$-abelian when $p\geq 5$;
\item $T(G)$ is not a $2$-group when $p-1$ is not a power of two or $r\geq2$.
\end{itemize}
This was first shown in \cite[Proposition 3.1]{Caranti3} and some specific examples were given in \cite[Section 5]{Caranti3}. 

\vspace{1.5mm}

In the case that $3\leq n_c\leq p-1$, we have $t\geq 1$, and from Theorem~\ref{main thm'} we see that $T(G)$ has a cyclic subgroup of order $\max\{1,p^{r-t}\}$, whence
\begin{itemize}
\item $T(G)$ is not a $2$-group when $r\geq t+1$.
\end{itemize}
However, the condition $r\geq t+1$ seems quite restrictive, and in particular it imposes that $r\geq2$ for $t\geq1$. Notice that $G$ has order at least $p^{n_c+1}$, and $G$ is said to be of \emph{maximal class} if its order is precisely $p^{n_c+1}$.

\begin{lem}\label{max class lem}Assume that $G$ is of maximal class and that $n_c\leq p$. Then, the quotient $G/Z(G)$ has exponent $p$, namely $r=1$.
\end{lem}
\begin{proof}See \cite[Theorem 9.5]{B book}.\end{proof}

All groups of order $p,p^2,p^3$ have nilpotency class at most $2$. For groups $G$ of order $p^4$, none of them can satisfy $n_c=3$ and $r\geq t+1$ by Lemma~\ref{max class lem}. As for groups $G$ of order $p^5,p^6$, it is known that they may be separated into $10$ and $43$  so-called \emph{isoclinism} families, respectively. In the two tables below, let us summarize the information of interest to us, which is found in \cite[Section 4]{James}. The list of groups of order $p^6$ given in \cite{James} is incomplete and the number of groups in each family stated below is taken from \cite[Table 2]{p6}.
\begin{longtable}{|c|c|c|c|c|}
\hline
Family&Number of groups& $n_c$& $(r,t)$ & $r\geq t+1?$
\\\hline
$3$ & $13$ & $3$ & $(1,1)$ & No\\
$6$ & $p+7$ & $3$ & $(1,1)$& No\\
$7$ & $5$ & $3$ & $(1,1)$& No\\
$8$ & $1$ & $3$ & $(2,1)$& \textbf{Yes}\\
$9$ & $\gcd(p-1,3)+2$ & $4$ & $(1,1)$& No\\
$10$ & $\gcd(p-1,3) +\gcd(p-1,4)+1$ & $4$ & $(1,1)$& No\\
\hline
\end{longtable}
\vspace{-5mm}
\begin{center}{\small The isoclinism families of groups of order $p^5$ and nilpotency class $n_c\geq3$ for $p\geq5$.}\end{center}
\begin{longtable}{|c|c|c|c|c|}
\hline
Family&Number of groups& $n_c$& $(r,t)$& $r\geq t+1?$
\\\hline
$3$ & $32$ & $3$ & $(1,1)$ & No\\
$6$ & $2p+21$ & $3$ & $(1,1)$ & No\\
$7$ & $21$ & $3$ & $(1,1)$ & No\\
$8$ & $p+5$ & $3$ & $(2,1)$ & \textbf{Yes}\\
$9$ & $3\gcd(p-1,3)+7$ & $4$ & $(1,1)$ & No\\
$10$ & $3\gcd(p-1,3) +3\gcd(p-1,4)+4$ & $4$ & $(1,1)$ & No\\
$16$ & $p+\gcd(p-1,3)+12$ & $3$ & $(1,1)$ & No\\
$17$ & $4p+\gcd(p-1,3)+30$ & $3$ & $(1,1)$ & No\\
$18$ & $3p+\gcd(p-1,3)+\gcd(p-1,4)+9$ & $3$ & $(1,1)$ & No\\
$19$ & $(3p^2+10p+21)/2$ & $3$ & $(1,1)$ & No\\
$20$ & $5p+\gcd(p-1,3)+\gcd(p-1,4)+13$ & $3$ & $(1,1)$ & No\\
$21$ & $(3p^2+4p+5)/2$ & $3$ & $(1,1)$ & No\\
$22$ & $7$ & $3$ & $(1,1)$ & No\\
$23$ & $p+4\gcd(p-1)+\gcd(p-1,4)+5$ & $3$ & $(1,1)$ & No\\
$24$ & $\gcd(p-1,3)+3$ & $3$ & $(1,1)$ & No\\
$25$ & $(p+3)/2$ & $3$ & $(2,1)$ & \textbf{Yes}\\
$26$ & $(p+3)/2$ & $3$ & $(2,1)$ & \textbf{Yes}\\
$27$ & $\gcd(p-1,3)+\gcd(p-1,4)+3$ & $3$ & $(1,1)$ & No\\
$28$ & $p$ & $3$ & $(2,1)$ & \textbf{Yes}\\
$29$ & $p$ & $3$ & $(2,1)$ & \textbf{Yes}\\
$30$ & $2\gcd(p-1,3)+4$ & $3$ & $(1,1)$ & No\\
$31$ & $7$ & $3$ & $(1,1)$ & No\\
$32$ & $5$ & $3$ & $(1,1)$ & No\\
$33$ & $6$ & $3$ & $(1,1)$ & No\\
$34$ & $3$ & $3$ & $(2,1)$ & \textbf{Yes}\\
$35$ & $\gcd(p-1,4)+2$ & $3$ & $(1,1)$ & No\\
$36$ & $\gcd(p-1,4)+\gcd(p-1,6)+1$ & $3$ & $(1,1)$ & No\\
$37$ & $\gcd(p-1,4)+4$ & $3$ & $(1,1)$ & No\\
$38$ & $p+\gcd(p-1,4)+\gcd(p-1,5)$ & $3$ & $(1,1)$ & No\\
$39$ & $p+\gcd(p-1,5)+\gcd(p-1,6)$ & $3$ & $(1,1)$ & No\\
$40$ & $\gcd(p-1,3)+2$ & $3$ & $(1,1)$ & No\\
$41$ & $\gcd(p-1,3)+1$ & $3$ & $(1,1)$ & No\\
$42$ & $p+1$ & $3$ & $(2,1)$ & \textbf{Yes}\\
$43$ & $p$ & $3$ & $(2,1)$ & \textbf{Yes}\\
\hline
\end{longtable}
\vspace{-5mm}
\begin{center}{\small The isoclinism families of groups of order $p^6$ and nilpotency class $n_c\geq3$ for $p\geq5$.}\end{center}

Hence, for $p\geq 5$, there are exactly one and $6p+12$, respectively, groups $G$ of order $p^5$ and $p^6$ with $n_c\geq 3$ and $r\geq t+1$. For both $i=5,6$, the ratio
\[ \frac{\#\{\mbox{groups $G$ of order $p^i$ with $n_c\geq 3$ and $r\geq t+1$}\}}{\#\{\mbox{all groups $G$ of order $p^i$ with $n_c\geq 3$}\}}\]
is quite small and tends to zero as $p$ increases. But for $i=6$, the numerator tends to infinity as $p$ increases, and so Theorem~\ref{main thm'} gives infinitely many new examples $G$ (with $3\leq n_c\leq p-1$) for which $T(G)$ is not a $2$-group.
%\begin{question}For a fixed value of $n_c$, is there a finite $p$-group $G$ of nilpotency class $n_c$ for which $3\leq n_c\leq p-1$ and $r\geq t+1$?\end{question}

\subsubsection{Large nilpotency class}\label{new sec}

The key ingredient to the proof of Theorem~\ref{main thm'} is Lemmas~\ref{formula lem} and~\ref{extra lem}. In view of Lemma~\ref{extra lem}, it seems natural to ask whether Theorem~\ref{main thm'} may be extended to the case when $G$ is only regular but $n_c\geq p$. We thank the referee for suggesting this. However, if $G$ is regular but $n_c\geq p$, then $\pi_n$ is reasonably well-behaved only when $n$ is a power of $p$, in which case $\pi_n$ is not bijective. In fact, by using the {\sc SmallGroup} library \cite{SGlibrary} in {\sc Magma} \cite{magma}, we found a counterexample.

\begin{example}\label{regular eg}Let $G=\mbox{{\sc SmallGroup}}(729,22)$. The following information may be easily calculated in {\sc Magma}.
\begin{itemize}
\item $G$ has exponent $27$.
\item $G$ has nilpotency class $3$, namely $n_c = 3$.
\item $G/Z(G)$ has exponent $9$, namely $r=2$.
\item $\gamma_3(G)$ has exponent $3$, namely $t=1$.
\item $\gamma_2(G)$ has order $9$.
\item $\gamma_2(G)$ is a cyclic group.
\end{itemize}
The last fact implies that $G$ is regular as a $3$-group, by \cite[(3.13)]{Suzuki book}. However, running the code below in {\sc Magma} yields the output $[1,8,10,17,19,26]$.
\begin{lstlisting}[
  mathescape,
  columns=fullflexible,
  basicstyle=\ttfamily,
]
  G:=SmallGroup(729,22);
  X:=[x:x in G];
  s:=X[82]; // a choice of $\sigma\in G$
  x:=X[244]; // a choice of $x\in G$
  y:=x^2;
  L:=[];
  for n in [n:n in [1..27]|GreatestCommonDivisor(n,3) eq 1] do
    nt:=InverseMod(n,27); // the $\widetilde{n}\in\bN$ in (3.2)
    Zx:=(x^(nt)*s^(-1))^n*s^n; // the value of $\zeta_{n,\sigma}(x)$ as in (3.2)
    Zy:=(y^(nt)*s^(-1))^n*s^n; // the value of $\zeta_{n,\sigma}(x^2)$ as in (3.2)
    if Zy eq (Zx)^2 then // test if $\zeta_{n,\sigma}(x^2) = \zeta_{n,\sigma}(x)^2$ holds
      Append(~L,n);
    end if;
  end for;
  L; // list of $n$ mod $27$ for which $\zeta_{n,\sigma}$ might be in $\Aut(G)$
\end{lstlisting}
This implies that there exists $\sigma\in G$ such that $\zeta_{n,\sigma}\not\in\Aut(G)$ for all $n\in\bZ$ coprime to $3$ with $n\not\equiv \pm1$ (mod $9$). By (\ref{pi n}) and Lemma~\ref{zeta lem}, it follows that $\pi_n\notin\NHol(G)$ for these $n$, whence $\kappa$ is not well-defined.

\vspace{1.5mm}

Let $m\in\bZ$. For any $x,y\in G$, by regularity we have
\[ \pi_{9m+1}(xy) = (xy)^{9m}(xy) = (x^9y^9\mu^9)^m(xy),\mbox{ where }\mu\in\gamma_2(\langle x,y\rangle),\]
and $\mu^9=1_G$ because $\gamma_2(G)$ has order $9$. Also, notice that $x^9,y^9\in Z(G)$ since $G/Z(G)$ has exponent $9$. It follows that
\[\pi_{9m+1}(xy) = (x^{9}y^{9})^mxy =x^{9m}y^{9m}xy = \pi_{9m+1}(x)\pi_{9m+1}(y). \]
This shows that 
\[\pi_{9m+1}\in\Aut(G)\mbox{ and }\pi_{9m-1}\equiv\pi_{-1}\mmod{\Aut(G)}.\]
Therefore, the power maps $\pi_n$ with $n\in\bZ$ coprime to $3$ only give rise to the subgroup $\langle\pi_{-1}\Hol(G)\rangle$ of order two in $T(G)$.

\vspace{1.5mm}

Nonetheless, by computing the size of (\ref{HH0}), which is the order of $T(G)$, we found in {\sc Magma} that $T(G)$ has order $18$ and hence is not a $2$-group.
\end{example}

Example~\ref{regular eg} shows that Theorem~\ref{main thm'} may not be generalized to all regular $p$-groups. It also shows that there is a (regular) $p$-group $G$ of nilpotency class at least $p$ for which $T(G)$ is not a $2$-group, and the elements in $T(G)$ of odd order do not arise from the power maps. It leads to the following questions.

\begin{question}Give more examples of $p$-groups $G$ of nilpotency class $n_c\geq p$ for which $T(G)$ is not a $2$-group.
\end{question}

\begin{question}Other than the power maps $\pi_n$ for $n$ coprime to $p$, what are other ways to construct elements in $T(G)$ of odd order?
\end{question}

A. Caranti \cite{Caranti3} used bilinear maps but his method is specific for $p$-groups of nilpotency class two. It would be interesting to see if there are other methods which work more generally, and in particular, figure out where the elements in $T(G)$ of odd order come from for $G$ with {\sc SmallGroup} ID equal to $(729,22)$, as well as those in (\ref{(n,i)}).

\section{Acknowledgments}

The author thanks the referee for helpful suggestions which led to the discussion in Subsection~\ref{new sec}.


\begin{thebibliography}{99}

\bibitem{B book}
Y. Berkovich, \emph{Groups of prime power order. Vol. 1}. With a foreword by Zvonimir Janko. De Gruyter Expositions in Mathematics, 46. Walter de Gruyter GmbH \& Co. KG, Berlin, 2008.

\bibitem{SGlibrary}
H. U. Besche, B. Eick, and E. A. O'Brien, \emph{A millennium project: constructing small groups}, Internat. J. Algebra Comput. 12 (2002), no. 5, 623--644.

\bibitem{magma}
W. Bosma, J. Cannon, and C. Playoust, \emph{The Magma algebra system. I. The user language}, J. Symbolic Comput., 24 (1997), 23--265.

\bibitem{Byott squarefree}
A. A. Alabdali and N. P. Byott, \emph{Counting Hopf-Galois structures on cyclic field extensions of squarefree degree}, J. Algebra 493 (2018), 1--19.

%\bibitem{By96}
%N. P. Byott, \emph{Uniqueness of Hopf-Galois structure of separable field extensions}, Comm. Algebra 24 (1996), no. 10, 3217--3228. Corrigendum, \emph{ibid}. no. 11, 3705.

%\bibitem{Byott pq}
%N. P. Byott, \emph{Hopf-Galois structures on Galois field extensions of degree $pq$}, J. Pure Appl. Algebra 188 (2004), no. 1--3, 45--57.

%\bibitem{Byott simple}
%N. P. Byott, \emph{Hopf-Galois structures on field extensions with simple Galois groups}, Bull. London Math. Soc. 36 (2004), no. 1, 23--29.

%\bibitem{Byott almost cyclic}
%N. P. Byott, \emph{Hopf-Galois structures on almost cyclic field extensions of $2$-power degree}, J. Algebra 318 (2007), no. 1, 351--371.

%\bibitem{Byott Childs}
%N. P. Byott and L. N. Childs, \emph{Fixed-point free pairs of homomorphisms and nonabelian Hopf-Galois structures}, New York J. Math. 18 (2012), 707--731.

%\bibitem{Byott abelian}
%N. P. Byott, \emph{Nilpotent and abelian Hopf-Galois structures on field extensions}, J. Algebra 381 (2013), 131--139.

%\bibitem{Byott soluble}
%N. P. Byott, \emph{Solubility criteria for Hopf-Galois structures}, New York J. Math 21 (2015), 883--903.

\bibitem{Caranti1}
A. Caranti and F. Dalla Volta, \emph{The multiple holomorph of a finitely generated abelian group}, J. Algebra 481 (2017), 327--347.

\bibitem{Caranti2}
A. Caranti and F. Dalla Volta, \emph{Groups that have the same holomorph as a finite perfect group}, J. Algebra 507 (2018), 81--102.

\bibitem{Caranti3}
A. Caranti, \emph{Multiple holomorphs of finite $p$-groups of class two}, J. Algebra 516 (2018), 352--372.

%\bibitem{Childs non-abelian}
%S. Carnahan and L. N. Childs, \emph{Counting Hopf-Galois structures on non-abelian Galois field extensions}, J. Algebra 218 (1999), no. 1, 81--92.

\bibitem{Childs book}
L. N. Childs, \emph{Taming wild extensions: Hopf algebras and local Galois module theory}. Mathematical Surveys and Monographs, 80. American Mathematical Society, Providence, RI, 2000.

%\bibitem{Childs EA}
%L. N. Childs, \emph{Elementary abelian Hopf Galois structures and polynomial formal groups}, J. Algebra 283 (2005), no. 1, 292--316.

%\bibitem{Childs PAMS}
%L. N. Childs, \emph{Fixed-point free endomorphisms and Hopf-Galois structures}, Proc. Amer. Math. Soc. 141 (2013), no. 4, 1255--1265.

%\bibitem{FSH}
%S. C. Featherstonhaugh, A. Caranti, and L. N. Childs, \emph{Abelian Hopf Galois structures on prime-power Galois field extensions}, Trans. Amer. Math. Soc. 364 (2012), no. 7, 3675--3684.

%\bibitem{gap}
%The GAP Group, GAP -- Groups, Algorithms, and Programming, Version 4.10.0; 2018. %(https://www.gap-system.org)

%\bibitem{G book}
%D. Gorenstein, \emph{Finite simple groups. An introduction to their classification}. University Series in Mathematics. Plenum Publishing Corp., New York, 1982.

%\bibitem{GP}
%C. Greither and B. Pareigis, \emph{Hopf-Galois theory for separable field extensions}, J. Algebra 106 (1987), 261--290.

\bibitem{Skew braces}
L. Guarnieri and L. Vendramin, \emph{Skew braces and the Yang-Baxter equation}, Math. Comp. 86 (2017), no. 307, 2519--2534.

%\bibitem{RG}
%R. X. Guralnick, \emph{Subgroups of prime power index in a simple group}, J. Algebra 81 (1983), no. 2, 304--311.

\bibitem{MHall book}
M. Hall Jr., \emph{The theory of groups}. The Macmillan Co., New York, N.Y. 1959.

\bibitem{Hall}
P. Hall, \emph{A contribution to the theory of groups of prime-power order}, Proc. London Math. Soc. (2) 36 (1934), 29--95. 

\bibitem{James}
R. James, \emph{The groups of order $p^6$ (p an odd prime)}, Math. Comp. 34 (1980), no. 150, 613--637. 

%\bibitem{Kohl98}
%T. Kohl, \emph{Classification of the Hopf-Galois structures on prime power radical extensions}, J. Algebra 207 (1998), no. 2, 525--546.

%\bibitem{Kohl 4p}
%T. Kohl, \emph{Groups of order $4p$, twisted wreath products and Hopf-Galois theory}, J. Algebra 314 (2007), no. 1, 42--74.

\bibitem{Kohl NHol}
T. Kohl, \emph{Multiple holomorphs of dihedral and quaternionic groups}, Comm. Algebra 43 (2015), no. 10, 4290--4304.

%\bibitem{book}
%G. A. Miller, H. F. Blichfeldt, and L. E. Dickson, \emph{Theory and applications of finite groups}. John Wiley \& Sons, New York, 1916.

\bibitem{Mills}
W. H. Mills, \emph{Multiple holomorphs of finitely generated abelian groups}, Trans. Amer. Math. Soc. 71 (1951), 379--392.

\bibitem{Miller}
G. A. Miller, \emph{On the multiple holomorphs of a group}, Math. Ann. 66 (1908), no. 1, 133--142.

%\bibitem{Quick}
%M. Quick, \emph{Probabilistic generation of wreath products of non-abelian finite simple groups}, Comm. Algebra 32 (2004), no. 12, 4753--4768.

\bibitem{squarefree}
M. Ram Murty and V. Kumar Murty, \emph{On groups of squarefree order}, Math. Ann. 267 (1984), no. 3, 299--309.

\bibitem{p6}
M. F. Newman, E. A. O'Brien, and M. R. Vaughan-Lee, \emph{Groups and nilpotent Lie rings whose order is the sixth power of a prime}, J. Algebra 278 (2004), no. 1, 38--401. 

\bibitem{Rob book}
D. J. S. Robinson, \emph{A course in the theory of groups}. Second edition. Graduate Texts in Mathematics, 80. Springer-Verlag, New York, 1996.

%\bibitem{Skew}
%A. Smoktunowicz and L. Vendramin, \emph{On skew braces (with an appendix by N. Byott and L. Vendramin}, J. Comb. Algebra 2 (2018), no. 1, 47--86. 

\bibitem{Suzuki book}
M. Suzuki, \emph{Group theory. II}. Translated from the Japanese. Grundlehren der Mathematischen Wissenschaften, 248. Springer-Verlag, New York, 1986.

%\bibitem{Tsang HG}
%C. Tsang, \emph{Non-existence of Hopf-Galois structures and bijective crossed homomorphisms}, J. Pure Appl. Algebra 223 (2019), no. 7, 2804--2821.

%\bibitem{Tsang char simple}
%C. Tsang, \emph{Hopf-Galois structures of isomorphic type on a non-abelian characteristically simple extension}, to appear in Proc. Amer. Math. Soc..

\bibitem{Tsang NHol}
C. Tsang, \emph{On the multiple holomorph of a finite almost simple group}, New York J. Math. 25 (2019), 949--963.

%\bibitem{Z thesis}
%K. N. Zenouz, \emph{On Hopf-Galois structures and skew braces of order $p^3$}, PhD Thesis, University of Exeter 2018.

\end{thebibliography}
\end{document}